\documentclass[10pt]{article}
\usepackage{amsmath, amsthm}\usepackage{enumerate}
\usepackage{amssymb}\usepackage{bbold}
\usepackage{color}
\usepackage[a4paper, margin=25mm]{geometry}
\numberwithin{equation}{section}
\theoremstyle{plain}
\newtheorem{theorem}{Theorem}
\newtheorem{assertion}{Assertion}
\newtheorem{lemma}{Lemma}[section]

\theoremstyle{definition}
\newtheorem{definition}{Definition}

\newtheorem{example}{Example}

\DeclareMathOperator{\convo}{\xrightarrow[]{o}}
\DeclareMathOperator{\convr}{\xrightarrow[]{(ru)}}

\large
\makeatletter
\renewcommand{\subsection}{\@startsection{subsection}{1}
{0pt}{3.25ex plus 1ex minus.2ex}{-1em}{\normalfont\normalsize\bf}}
\makeatother
\begin{document}

\title{{\bf On the extension of one-parameter operator semigroups to completions of  Archimedean vector lattices}}
\maketitle
\author{\centering{{Eduard Emelyanov$^{1}$\\ 
\small $1$ Sobolev Institute of Mathematics, Novosibirsk, Russia}

\abstract{Extensions of one-parameter operator semigroups on Archimedean vector lattices to their 
order/ru-completions are studied. Existence and uniqueness of the extension to the ru-completion 
is established in the class of positive semigroups. An extension theorem for positive ru-continuos semigroups 
on vector lattices with property (R) is proved. This theorem allows one to abandon the requirement of ru-completeness 
in various results concerning positive ru-continuous semigroups.}


\large


\section{Introduction and preliminaries}

\hspace{4mm}
Under a semigroup on a vector space $X$ we understand a family $(T_s)_{s\ge 0}$ of linear operators on $X$, 
enjoying the semigroup property $T_{s+t}=T_sT_t$ \ ($s,t\in\mathbb{R}_+$), such that $T_0=I_X$ is 
the identity operator on $X$. Semigroups are in the close relationship with the Cauchy problem
$\frac{d}{dt}f=Af$,\ $f(0)=f_0\in X$, where $f_0$ is an initial condition and $A$ is a densely defined
(in the sense of some convergence) operator on $X$. In general, a semigroup $(T_s)_{s\ge 0}$ on $X$ describes
an evolution of states in the space $X$ which may face a lack of an appropriate topology for 
constructing a necessary in the Cauchy problem operator $A$. Note that the theory of strongly continuous (or $C_0$-) semigroups on Banach spaces
is a well investigated area of functional analysis. In the setting of locally convex spaces, the theory of semigroups is also rather developed (it goes back to 
works of I.~Miyadera, H.~Komatsu, K.~Singbal-Vedak, T.~Komura \cite{Miya1959,Koma1964,S-V1965,Komu1968}, and other authors).

The recent progress in this direction is achieved in works of M.~Kandi\'{c}, M.~Kaplin, M.~Kramar Fijav\v{z}, and J.~Gl\"{u}ck \cite{KK2020,KK-F2020,GK2024}.
It these works, the relatively uniform topology is considered, that need not to be locally convex (for instance on $L^p(\mathbb{R})$, $0<p<1$).
In particular, in \cite{KK-F2020} a generator (constructed by use of the ru-convergence) 
of a semigroup on a relatively uniformly complete vector lattice is investigated. Y.\,A.~Dabboorasad, M.\,A.\,A.~Marabeh, and the author
in Theorem~5 of \cite{DEM2018} established that the ru-convergence in an Archimedean vector lattice is described by a linear topology
if and only if the vector lattice has a strong order unit. Therefore, in the study of a semigroup $(T_s)_{s\ge 0}$ on 
a vector lattice $X$ without a strong order unit, it is necessary to use ru-convergence rather than topological arguments.

It is of a considerable interest to build a generator of the extension of a semigroup which is defined on a dense, in an appropriate sense, sublattice. 
In the case of norm convergence, the extension preserving strong continuity is trivial for semigroups which are bounded in the operator norm 
in a neighborhood of zero. 
In the present paper we consider extendability of a semigroup on an Archimedean vector lattice
to its order/ru- completion and study when such an extension preserves continuity at zero
with respect to the order/ru- convergence.

Theorem \ref{thm-1} states the existence and uniqueness of an extension of a semigroup $(T_s)_{s\ge 0}$ of order continuous
operators from an Archimedean vector lattice $X$ to its order completion $X^\delta$.
Theorem \ref{thm-2} gives the existence and uniqueness of an extension of a positive semigroup on $X$ to its ru-completion $X^r$.
The main result of the paper, Theorem \ref{partial}, establishes extendability of positive semigroups to $X^r$ preserving continuity with
respect to ru-convergence on a vector lattice $X$ with the property $\text{\rm(R)}$.
Theorem \ref{partial} allows dropping the ru-completeness in principal results of work \cite{KK-F2020} 
for a rather broad class of semigroups.

In what follows, vector spaces are supposed to be real and operators linear. By latter $X$ we denote an Archimedean vector lattice, 
by $X^\delta$ the order completion of $X$, and by $X^r$ 
the completion with respect to convergence with regulator. 
For further unexplained symbols, notions, and terminology we refer to \cite{Vul1967,LZ1971,Kus2003,E2007,E2023}.

\section{On the extension of a semigroup to the order completion}\label{s2}

\hspace{4mm}
In order to extend a semigroup $(T_s)_{s\ge 0}$ on $X$ to a semigroup $(T_s^\delta)_{s\ge 0}$ on $X^\delta$,
it is necessary to extend every $T_s$ to an operator $T_s^\delta$ on $X^\delta$. Due to the well known Veksler result \cite{Vek1960},
it suffices to assume that the operators $T_s$ are regular and order continuous. 

The regularity condition on operators $T_s$ is not too restrictive, because the order continuity of
$T_s:X\to X$ is equivalent to the order continuity of $T_s:X\to X^\delta$ 
(as $X$ is a regular sublattice of $X^\delta$). 

Since $X^\delta$ is Dedekind complete, the operator $T_s:X\to X^\delta$ is regular if and only if
$T_s:X\to X^\delta$ is order bounded, which is in turn equivalent to the order boundedness of $T_s:X\to X$.

Furthermore, even the order boundedness of $T_s:X\to X$ can be easily dropped due to the
Abramovich -- Sirotkin theorem on automatic order boundedness of order continuous operators \cite[Theorem 2.1]{AS2005}.
Note that the extensions $T_s^\delta$ are unique and order continuous on $X^\delta$.
Summarizing, we obtain.

\begin{theorem}\label{thm-1}
Every semigroup $(T_s)_{s\ge 0}$ of order continuous operators on $X$ pos\-sesses a unique extension 
to the semigroup $\big(T_s^\delta\big)_{s\ge 0}$ of order continuous operators on $X^\delta$.
\end{theorem}

\hspace{4mm}
It seems to be important to find supplementary conditions under which
this extension preserves order continuity/ru-continuity of the semigroup $(T_s)_{s\ge 0}$, namely the follo\-wing properties.

\begin{definition}\label{thm-1}
A semigroup $(T_s)_{s\ge 0}$ has the property
\begin{itemize}
\item[\rm\bf oc]\ \
if, for every $x\in X$ and $s\ge 0$, we have $T_{s+h}x\convo T_sx$ as $h\to 0$;
\item[\rm\bf oc$_0$]\ \
if, for every $x\in X$ we have $T_hx\convo x$ as $h\downarrow 0$;
\item[\rm\bf ruc]\ \
if, for every $x\in X$ and $s\ge 0$ we have $T_{s+h}x\convr T_sx$ as $h\to 0$;
\item[\rm\bf ruc$_0$]\ \
if, for every $x\in X$ we have $T_hx\convr x$ as $h\downarrow 0$.
\end{itemize}
\end{definition}

\noindent
It can be easily seen 
$$
   \text{\rm\bf ruc}\Longrightarrow\text{\rm\bf ruc}_0;\ \ \ \text{\rm\bf oc}\Longrightarrow\text{\rm\bf oc}_0;\ \ \ 
   \text{\rm\bf ruc}\Longrightarrow\text{\rm\bf oc};\ \ \ \text{\rm\bf ruc}_0\Longrightarrow\text{\rm\bf oc}_0.
$$
\noindent
We shall return to question on reverse of the first implication in the last section.

\section{On extension of a semigroup to the \text{\rm ru}-completion}\label{s3}

\hspace{4mm}
Consider a semigroup $(T_s)_{s\ge 0}$ on $X$. For its extension to a semigroup $\left(T^{(r)}_s\right)_{s\ge 0}$ on $X^r$,
we need to extend each $T_s$ to an operator $T^{(r)}_s$ on $X^r$. First, we  discuss shortly the structure of $X^r$.
The idea of the construction of \text{\rm ru}-completion goes back to paper \cite{Vek1969} of A.I.~Veksler.
The first detailed presentation of the construction  of \text{\rm ru}-completion is contained in the recent papers 
of S.G.~Gorokhova and the author \cite{E2023,EG2024}. 

\smallskip
A vector sublattice $Z$ of an Archimedean vector lattice $Y$ is called \text{\rm ru}-{\em complete in} $Y$ if, for every \text{\rm ru}--Cauchy 
net $(x_\alpha)$ in $Z$ with a regulator $u\in Y$ there exist $x\in Z$ and $w\in Y$ such that $x_\alpha\convr x(w)$. 
Then $x_\alpha\convr x(u)$ by \cite[Lemma 1]{EG2024}. Moreover, the limit $x$ is unique because $X$ is Archimedean. 
Note that the \text{\rm ru}-completeness is sequential, namely the sublattice $Z$ is \text{\rm ru}-complete in $Y$ 
if and only if, for each \text{\rm ru}-Cauchy sequence $(x_n)$ in $Z$ with regulator $u\in Y$ we have that
$x_n\convr x(u)$ for some $x\in Z$ and $w\in Y$.
By \cite[Lemma 2]{EG2024}, we have

\begin{assertion}\label{FUFCL-lemma-2}
The following conditions are equivalent.
\begin{itemize}
\item[$i)$]\
$X$ is \text{\rm ru}-complete in itself.
\item[$ii)$]\
Each principal order ideal $I_z$ of $X$ is complete in the norm\\
$\|x\|_z:=\inf\{\lambda\in\mathbb{R}:|x|\le\lambda|z|\}$.
\item[$iii)$]\
$I_z$ is \text{\rm ru}-complete in itself for every $z\in X$.
\end{itemize}
\end{assertion}

\noindent
By the brothers Kreins -- Kakutani theorem, each $I_z$ from Assertion \ref{FUFCL-lemma-2}
is lattice isomorphic to $C(K)$ for some Hausdorff compact space $K$.

\begin{definition}\label{X-r}
(Cf., \cite[Definition 1]{EG2024})
An \text{\rm ru}-complete in itself vector lattice $F$ is called an
\text{\rm ru}-{\it completion} of $X$ if there exists a lattice embedding $J: X\to F$ such that, 
for each \text{\rm ru}-complete in itself vector lattice $Y$ and each lattice homomorphism $T:X\to Y$ 
there exists a unique lattice homomorphism $S: F\to Y$ satisfying $S\circ J=T$. 
\end{definition}

\hspace{4mm}
The \text{\rm ru}-completion $X^r$ exists and is unique up to a lattice isomorphism.
It can be represented as the intersection of all \text{\rm ru}-complete in itselves vector sublattices 
of some \text{\rm ru}-complete in itself vector lattice $Y$ containing $X$, e.g., one may take $Y=X^\delta$
(see, \cite[Lemma 1]{E2023} or \cite[Proposition 1]{EG2024}). It follows from \cite[Lemma 4]{EG2024},  

\begin{assertion}\label{asser01}
Let $X$ be an Archimedean vector lattice. The following holds.
\begin{itemize}
\item[$i)$]
$X^r=\bigcup\limits_{\gamma\in{\text{\rm Ord}}}X_\gamma=X_{\omega_1}$, moreover 
\item[$ii)$]
$X_1:=X$, 
\item[$iii)$]
$X_{\beta+1}:=\big\{x\in Y: x_n \convr x(x_1),\ for \ a \ sequence \ (x_n)\ \text{\rm in}\ X_\beta\big\}$, \ and \
\item[$iv)$]
$X_\gamma:=\bigcup\limits_{\beta<\gamma}X_{\beta}$ if $\gamma$ is a limit ordinal.
\end{itemize}
\end{assertion}

Now, we turn to the question of extendability of a semigroup $(T_s)_{s\ge 0}$ on $X$ to a semigroup
$(T_s^r)_{s\ge 0}$ on $X^r$. Note that the question has a positive solution for each semigroup
of lattice homomorphisms by the definition of $X^r$. The next theorem extends this fact to positive semigroups.

\begin{theorem}\label{thm-2}
Every positive semigroup $(T_s)_{s\ge 0}$ on $X$ has a unique extension to a positive semigroup $(T_s^r)_{s\ge 0}$ on $X^r$. 
\end{theorem}

\begin{proof}
Let $T^{(1)}_s=T_s$ for all $s\ge 0$ and suppose that the extension $\big(T^{(\beta)}_s\big)_{s\ge 0}$ is already constructed. 
In order to obtain the extension$\big(T^{(\beta+1)}_s\big)_{s\ge 0}$ on $X_{\beta+1}$, let $s\ge 0$ and $x\in X_{\beta+1}$.  
Pick a sequence $x_n\convr x(x_1)$ in $X_\beta$. For each $k\in\mathbb{N}$, there exists $n_k$ 
such that $|x_n-x|\le\frac{1}{k}|x_1|$ for $n\ge n_k$. Then $|x_n-x_m|\le\frac{2}{k}|x_1|$ for $n,m\ge n_k$. 
Since $T^{(\beta)}_s\ge 0$ then $T^{(\beta)}_s(x_n-x_m)\in\frac{2}{k}\big[-T^{(\beta)}_s|x_1|,T^{(\beta)}_s|x_1|\big]$ for $n,m\ge n_k$,
and hence the sequence $\big(T^{(\beta)}_s{x_n}\big)$ is \text{\rm ru}-Cauchy in $X^r$ with regulator $T^{(\beta)}_s|x_1|$. 
Since $X^r$ is \text{\rm ru}-complete in itself, there exists $y\in X^r$ such that $T^{(\beta)}_s{x_n}\convr y(w)$ with regulator $w\in X^r$. 
By \cite[Lemma 1]{EG2024}, $T^{(\beta)}_s{x_n}\convr y\big(T^{(\beta)}_s|x_1|\big)$, and hence $y\in X_{\beta+1}$. 
Set $T^{(\beta+1)}_sx:=y$. We omit a routing check of the following two facts.
\begin{itemize}
\item[-]
A vector $T^{(\beta+1)}_sx\in X_{\beta+1}$ does not depend on the choice of a sequence $(x_n)$ in $X_\beta$ such that $x_n\convr x(x_1)$.
\item[-]
$T^{(\beta+1)}_s$ is a positive linear operator on $X_{\beta+1}$.
\end{itemize}

Let $T'^{(\beta+1)}_s$ be another positive linear extension of $T^{(\beta)}_s$ on $X_{\beta+1}$.
Take an arbitrary $x\in X_{\beta+1}$ and find a sequence $(x_n)$ in $X_\beta$ 
and an increasing sequence $(n_k)$ such that $|x_n-x|\le\frac{1}{k}|x_1|$ for $n\ge n_k$. Then 
$$
   |T^{(\beta)}_sx_n-T^{(\beta+1)}_sx|+|T'^{(\beta)}_sx_n-T'^{(\beta+1)}_sx|\le\frac{1}{k}\Big(T^{(\beta)}_s|x_1|+T'^{(\beta)}_s|x_1|\Big)\quad (n\ge n_k).
   \eqno(1)
$$
Since $(x_n)$ is contained in $X_\beta$ then $T^{(\beta)}_sx_n=T'^{(\beta)}_sx_n$ for all $n$. Therefor, (1) implies
$$
   T^{(\beta)}_sx_n\convr T^{(\beta+1)}_sx\Big(T^{(\beta)}_s|x_1|+T'^{(\beta)}_s|x_1|\Big)
$$
and
$$ 
   T'^{(\beta)}_sx_n\convr T'^{(\beta+1)}_sx\Big(T^{(\beta)}_s|x_1|+T'^{(\beta)}_s|x_1|\Big).
$$ 
As $X_{\beta+1}$ is Archimedean, $T^{(\beta+1)}_sx=T'^{(\beta+1)}_sx$.
Thus, each $T^{(\beta)}_s$ has a unique extension on $X_{\beta+1}$.
Let $s_1,s_2\ge 0$. Since the operators $T^{(\beta+1)}_{s_1}T^{(\beta+1)}_{s_1}$ and $T^{(\beta+1)}_{s_1+s_2}$
both extend $T^{(\beta)}_{s_1+s_2}$ from $X_{\beta}$ on $X_{\beta+1}$, then $T^{(\beta+1)}_{s_1}T^{(\beta+1)}_{s_1}=T^{(\beta+1)}_{s_1+s_2}$,
and hence $\big(T^{(\beta+1)}_s\big)_{s\ge 0}$ is a semigroup.

\smallskip
Now, let $\gamma$ be a limit ordinal such that, for each $\beta<\gamma$ the semigroup $\big(T^{(\beta)}_s\big)_{s\ge 0}$ 
is a unique positive extension of the semigroup $(T_s)_{s\ge 0}$ on $X_{\beta}$.
Take an arbitrary $x\in X_\gamma=\bigcup_{\beta<\gamma}X_{\beta}$.
Then $x\in X_\beta$ for some $\beta<\gamma$. Thus, the elements $T^{(\beta)}_sx$ are already defined for $s\ge 0$. 
Let $T^{(\gamma)}_sx:=T^{(\beta)}_sx$ for some $\beta<\gamma$. It is easy to see that
\begin{itemize}
\item[-]
$T^{(\gamma)}_sx$ does not depend on the choice of $\beta<\gamma$;
\item[-]
The mappings $T^{(\gamma)}_s:X_\gamma\to X_\gamma$ are positive linear operators on $X_\gamma$; 
\item[-]
The extension $\big(T^{(\gamma)}_s\big)_{s\ge 0}$ on $X_\gamma$ is unique.
\end{itemize}
Since $X^r=X_{\omega_1}$, the proof is complete.
\end{proof}

\hspace{4mm}
It is still unknown whether or not the condition that the semigroup in Theorem \ref{thm-2} is positive
can be weakened to the order boundedness or to the order continuity of all operators $T_s$ of the semigroup. 
In the second case, Theorem \ref{thm-1} provides a unique extension of $(T_s)_{s\ge 0}$ to the semigroup $(T_s^\delta)_{s\ge 0}$ 
of order continuous (and hence regular) operators on $X^\delta$. Accordingly to \cite[Proposition 1]{EG2024}, 
$X^r$ can be identified with the intersection of all \text{\rm ru}-complete in itself vector sublattices of $X^\delta$.
Remark that, under the next additional assumption
\begin{itemize}
\item[$(*)$]\ 
$T_s^\delta(X^r)\subseteq X^r$ for all $s\ge 0$,
\end{itemize}
an extension of the semigroup $(T_s)_{s\ge 0}$ to $X^r$ is obtained as a restriction of the extension $(T_s^\delta)_{s\ge 0}$, namely
$T_s^\delta|_{X^r}=T_s^r$ for all $s\ge 0$. It is interesting to find conditions under which $(*)$ holds.

\section{On a preserving of ru-continuity extension of a positive semigroup to the ru-completion}\label{s4}

\hspace{4mm}
Here, we are going to prove Theorem \ref{partial} on a preserving of ru-continuity extension,
that is the main result of the paper. Recall the following classical definition from the Vulikh book \cite{Vul1967}.

\begin{definition}\label{R-property}(See, \cite[Definition VI.5.1]{Vul1967})
A vector lattice $X$ satisfies condition \text{\rm(R)} $($shortly, $X\in\text{\rm(R)}$$)$ if, for each sequence $(y_k)$ in $X$ 
there exist $y\in X$ and a sequence $(\lambda_k)$ of positive scalars such that $|\lambda_k y_k|\le y$ for all $k$.
\end{definition}

\noindent
Accordingly to \cite[Definition 70.1]{LZ1971}, the condition \text{\rm(R)} is also called the $\sigma$-property.
By \cite[Theorem 71.5]{LZ1971}, the vector lattice $L^0(\Omega,\Sigma,\mu)$ of classes of $\mu$-almost everywhere equal 
real-valued functions on $\Omega$ satisfies the condition \text{\rm(R)} whenever the measure $\mu$ is $\sigma$-finite.
We need the following characterization of the condition \text{\rm(R)}.

\begin{lemma}\label{R-prop}
Let $X$ be an Archimedean vector lattice. The following conditions are equivalent.
\begin{itemize}
\item[$i)$]\
$X\in\text{\rm(R)}$.
\item[$ii)$]\
Any countably generated order ideal of $X$ is contained in a principal ideal.
\item[$iii)$]\
Any countable family of \text{\rm ru}-convergent sequences in $X$ has a common regulator.
\item[$iv)$]\
Any countable family of \text{\rm ru}-convergent nets in $X$ has a common regulator.
\item[$v)$]\
$X^\delta\in\text{\rm(R)}$.
\item[$vi)$]\
$X^r\in\text{\rm(R)}$.
\item[$vii)$]\
$X_\beta\in\text{\rm(R)}$ for every ordinal $\beta$.
\item[$viii)$]\
$X_\beta\in\text{\rm(R)}$ for some ordinal $\beta$.
\end{itemize}
\end{lemma}

\begin{proof}
The equivalentness $i)\Longleftrightarrow ii)$ is trivial, whereas $ii)\Longleftrightarrow iii)\Longleftrightarrow iv)$ are proved
for example in \cite[Proposition~5.2 and Corollary~5.3]{KK2020}.

\smallskip
Let us consider $X$ as a sublattice of $X^r$, and $X^r$ as a sublattice of $X^\delta$.

\smallskip
$i)\Longrightarrow v)$\
Let $(y_n)$ be a sequence of elements of $X^\delta$. Since $X$ is a majorizing sublattice in $X^\delta$, 
there is a sequence $(g_n)$ in $X$ such that $|y_n|\le g_n$ for every $n$.
Since $X\in\text{\rm(R)}$, there exist $g\in X\subseteq X^\delta$ and a sequence $(\lambda_n)$ in $\mathbb{R}\setminus\{0\}$ 
such that $|\lambda_n g_n|\le g$ for all $n$. 
The inequalities $|\lambda_n y_n|\le|\lambda_n g_n|\le g$ hold for every $n$. Therefore, $X^\delta\in\text{\rm(R)}$.

\smallskip
$v)\Longrightarrow vi)$\
Let $(y_n)$ be a sequence in $X^r$. Since $X$ is a majorizing sublattice in $X^\delta$,
it is also majorizing in $X^r$. So, there is a sequence $(g_n)$ in $X\subseteq X^\delta$ such that $|y_n|\le g_n$ for all $n$.
Since $X^\delta\in\text{\rm(R)}$, there exist $g\in X^\delta$ and a sequence $(\lambda_n)$ in $\mathbb{R}\setminus\{0\}$ 
such that $|\lambda_n g_n|\le g$ for every $n$. Since $X$ is a majorizing sublattice in $X^\delta$, we may assume $g\in X\subseteq X^r$.
Consequently, $X^r\in\text{\rm(R)}$.

\smallskip
$vi)\Longrightarrow vii)$\ 
This implication follows from the item $i)$ of Assertion \ref{asser01}.

\smallskip
$vii)\Longrightarrow viii)$\ 
It is trivial.

\smallskip
$viii)\Longrightarrow i)$\
Let $X_{\beta_0}\in\text{\rm(R)}$. Take a sequence $(y_n)$ in $X\subseteq X_{\beta_0}$.
Since $X_{\beta_0}\in\text{\rm(R)}$, there exist $g\in X_{\beta_0}$ and a sequence $(\lambda_n)$ in $\mathbb{R}\setminus\{0\}$ 
such that $|\lambda_n y_n|\le g$ for all $n$. Since $X$ is a majorizing sublattice in $X^r$, it is also majorizing in $X_{\beta_0}$.
Thus, we may assume $g\in X$. Consequently, $X\in\text{\rm(R)}$.
\end{proof}

\noindent
By \cite[Example 5.5]{KK2020}, the vector lattice $\text{\rm Lip}(\mathbb{R})$ is not  
\text{\rm ru}-complete in itself and has a strong order unit. Therefore, it satisfies the condition \text{\rm(R)},
due to Lemma \ref{R-prop}. Consider one more example of a vector lattice satisfying the condition \text{\rm(R)}
which is not \text{\rm ru}-complete in itself.

\begin{example}\label{R-example1}
Let $X$ be a vector lattice of continuous real-valued functions on $\mathbb{R}$ such that
each of function in $X$ takes only finitely many values and has at most countable support.
Clearly, $X$ has no strong order unit. Take a sequence $(f_k)$ in $X$ and let $\Omega$ be the union
of supports of all $f_k$ of the sequence. It is clear that the indicator function ${\bold 1}_\Omega$ of $\Omega$ 
belongs to $X$ and there exists a sequence $(\lambda_k)$ of non-zero reals such that $|\lambda_k f_k|\le{\bold 1}_\Omega$ for all $k$. 
Therefore, $X\in\text{\rm(R)}$. It is also easy to see that $X$ is not \text{\rm ru}-complete in itself.
\end{example}

The following extension theorem is the main result of the paper.

\begin{theorem}\label{partial}
Let $(T_s)_{s\ge 0}$ be a positive semigroups on an Archimedean vector lattice $X\in\text{\rm(R)}$
such that $(T_s)_{s\ge 0}\in\text{\rm\bf ruc}_0$. Then $(T^r_s)_{s\ge 0}\in\text{\rm\bf ruc}$.
\end{theorem}

\begin{proof}
We may assume that $X$ is a sublattice of $X^r$. By Theorem \ref{thm-2}, the positive semigroup $(T_s^r)_{s\ge 0}$ is the unique 
extension of $(T_s)_{s\ge 0}$ to $X^r$. It follows from \cite[Theorem 5.4]{KK2020} 
and Lemma \ref{R-prop} that $X_\beta$ has the property \text{\rm(D)} (see Definition 3.5 in \cite{KK2020}) for every ordinal $\beta$.

We arrange the rest of the proof by use of transinite induction.
By \cite[Proposition 3.5]{KK2020}, $\big(T^{(1)}_s\big)_{s\ge 0}=(T_s)_{s\ge 0}\in\text{\rm\bf ruc}$ on $X_1=X$. 
Assume $\big(T^{(\beta)}_s\big)_{s\ge 0}$ has the property \text{\rm\bf ruc} on $X_\beta$ for all $\beta<\gamma$.

\smallskip
{\bf I}.\
Consider the case of a non-limit ordinal, $\gamma=\beta+1$. 
Take arbitrary $s\ge 0$ and $x\in X_{\beta+1}$. Since $X$ is a majorizing sublattice in $X_{\beta+1}$,
there exists $y\in X$, $|x|\le y$. As $(T_s)_{s\ge 0}\in\text{\rm\bf ruc}_0$ on $X$,
\cite[Proposition 3.4]{KK2020} implies $\{|T_ty|: 0\le t\le s\}\subseteq[0,z]$ for some $z\in X_+$. Then
$$
   \big|T^{(\beta+1)}_tx\big|\le T^{(\beta+1)}_t|x|\le T^{(\beta+1)}_ty=T_ty\le z\in X\subseteq X_{\beta+1} \ \ \ \ \ \ (0\le t\le s).
   \eqno(2)
$$
Set $D=X_\beta$ in \cite[Theorem 5.7]{KK2020}. Since $(T^{(\beta)}_s)_{s\ge 0}\in\text{\rm\bf ruc}_0$ then
$T^{(\beta+1)}_hg=T^{(\beta)}_hg\convr g$ as $h\downarrow 0$ for every $g\in X_\beta=D$.
In view of (2), the set $\big\{T^{(\beta+1)}_tx: 0\le t\le s\big\}$ is order bounded in $X_{\beta+1}$
for every $s\ge 0$ and $x\in X_{\beta+1}$. Applying \cite[Theorem 5.7]{KK2020}, we conclude that
$\big(T^{(\beta+1)}_s\big)_{s\ge 0}\in\text{\rm\bf ruc}$.

\smallskip
{\bf II}.\
Now, let $\gamma$ be a limit ordinal. Take an arbitrary $x\in X_\gamma$.
Then $x\in X_\beta$ for some $\beta<\gamma$. Since $\big(T^{(\beta)}_s\big)_{s\ge 0}\in\text{\rm\bf ruc}$
then $T^{(\gamma)}_hx=T^{(\beta)}_hx\convr x$ for $h\downarrow 0$, and hence   
$\big(T^{(\gamma)}_s\big)_{s\ge 0}\in\text{\rm\bf ruc}$.

\smallskip
Because $X^r=X_{\omega_1}$ by Assertion \ref{asser01}, the proof is complete.
\end{proof}

\smallskip
In the paper \cite{KK-F2020} the generator of a positive semigroup with property $\text{\rm\bf ruc}$ on
an \text{\rm ru}-complete in itself vector lattice is constructed. Assertion \ref{FUFCL-lemma-2} 
allows to simplify the arguments from \cite{KK-F2020}, by restricting the consideration to 
continuous functions on a compact Hausdorff space, whereas Theorem \ref{partial} 
is letting to avoid the \text{\rm ru}-completeness of underlying lattices in the rather broad class of 
semigroups on vector lattices satisfying the condition $\text{\rm(R)}$.

\end{document}